\documentclass[reqno]{amsart}

\usepackage{amsmath,amsthm,amsfonts, eucal, amssymb}
\usepackage[english]{babel}
\usepackage{graphicx,pdflscape}
\usepackage{layout}
\usepackage[all]{xy}
\usepackage{hyperref}
\usepackage[font=small,labelsep=none]{caption}
\usepackage{subcaption}
\usepackage{pinlabel}
\usepackage[all]{xy}

\newtheorem{thm}{Theorem}

\newtheorem{prop}[thm]{Proposition}

\newtheorem{lem}[thm]{Lemma}

\theoremstyle{definition}
\newtheorem{defn}[thm]{Definition}

\newtheorem{rem}[thm]{Remark}

\title{Knot exteriors with all possible meridional essential surfaces}

\author{Jo\~{a}o M. Nogueira}

\begin{document}

\begin{abstract}
We show the existence of infinitely many knot exteriors where each of which contains meridional essential surfaces of any genus and (even) number of boundary components. That is, the compact surfaces that have a meridional essential embedding into a knot exterior have a meridional essential embedding into each of these knot exteriors. From these results, we also prove the existence of a hyperbolic knot exterior in some 3-manifold for which there are meridional essential surfaces of independently unbounded genus and number of boundary components.

\end{abstract}

\maketitle

\section{Introduction}
Surfaces have an important role in the understanding of $3$-manifold topology.
This paper is concerned with the interesting phenomenon of certain knot exteriors having essential surfaces of arbitrarily large Euler characteristics. The first examples of knots with this property were given by Lyon \cite{Lyon}, where he proves the existence of knot exteriors where each of which has closed essential surfaces of arbitrarily high genus. Other examples were later obtained, for instance, by Oertel \cite{Oertel}, and, more recently, by Li \cite{Li} or by Eudave-Mu\~noz and Neumann-Coto \cite{Munoz-Coto}. Similarly, the author proved in \cite{Nogueira1} the existence of prime knot exteriors such that each contains meridional essential surfaces with two boundary components and arbitrarily high genus. On the other hand, one might wonder if the unbounded Euler characteristics of essential surfaces in a knot exterior can be from the number of boundary components instead of the genus. That is, is there a knot exterior with compact essential surfaces with arbitrarily many boundary components? This is in fact the case, as shown by the examples given by the author \cite{Nogueira2}, with meridional essential planar surfaces. The problem addressed in this paper is whether the arbitrarily large Euler characteristic can be obtained from independently unbounded genus and number of boundary components. Theorem \ref{main} answers this question affirmatively.    
\begin{thm}\label{main}
There are infinitely many knots each of which has in its exterior meridional essential surfaces of any genus and $2n$ boundary components for any $n\geq 1$. Moreover, the collection can be made of prime knots, naturally excluding the existence of meridional essential annuli in their exteriors.   
\end{thm}

\noindent Note that only the compact surfaces with an odd number of boundary components cannot be embedded as meridional essential surfaces into any knot exterior in $S^3$. In this regard, the knots in the statement of Theorem \ref{main} have in their exteriors all possible compact surfaces embedded as meridional essential surfaces.\\
Each knot exterior of Theorem \ref{main} also has closed essential surfaces of unbounded genus. In fact, from \cite{CGLS}, at least one swallow-follow surface obtained  
from each meridional essential surface, with two boundary components, in Theorem \ref{main} is of higher genus and also essential in the exterior of the respective knot.\\
There are many examples of knots that do not have the properties as in the theorems above, with the most obvious among these being small or meridionally small knots. Well known examples of classes of small knots are the torus knots, the $2$-bridge knots \cite{Hatcher-Thurston}, and Montesinos knots with length three \cite{Oertel2}, among other examples. One particularly interesting result, in contrast with the theorems in this paper, is one by Menasco \cite{Menasco} stating that for a fixed number of boundary components there are finitely many meridional essential surfaces in the exterior of a prime alternating link. In particular, for a fixed number of boundary components there is a bound on the genus for meridional essential surfaces. Hence, the knots of Theorem \ref{main} are not alternating.

\noindent In Theorem \ref{hyperbolic}, we show that hyperbolicity of the ambient space is not an obstruction to the existence of meridional essential surfaces of independently unbounded genus and number of boundary components in a $3$-manifold with torus boundary.
 
\begin{thm}\label{hyperbolic}
There are infinitely many hyperbolic $3$-manifolds with torus boundary each of which has meridional essential surfaces of independently unbounded genus and number of boundary components. 
\end{thm}

\noindent The paper is organized as follows: In section 2 of this paper we present a construction of knots used in the paper and prove some of their properties. For the construction we use satellite knots together with handlebody-knots of genus two. In section 3 we show a process to obtain knot exteriors with meridional essential surfaces of arbitrarily many boundary components as in Proposition \ref{boundaries}, and use the knots from the main theorem of \cite{Nogueira1} to prove Theorem \ref{main}. The main methods are classical in $3$-manifold topology; we use innermost curve arguments and branched surface theory. In section 4 we prove Theorem \ref{hyperbolic} using classical results in hyperbolic manifolds and degree-one maps. Throughout the paper we work in the smooth category, all knots are assumed to be in $S^3$, unless otherwise stated, and all (sub)manifolds are assumed to be orientable and in general position.

\section{A construction of knots.}

A common method to construct knots is through the construction of satellite knots. We start considering a knot $K_p$ in a solid torus $T$, that we refer to as the \textit{pattern knot}. The solid torus $T$ is embedded in $S^3$ by the map $\sigma: T\rightarrow S^3$ where the core of $\sigma(T)$ has image a knot $K_c$ that is called the \textit{companion knot}. The knot $\sigma(K_p)$ is called a \textit{satellite knot} of $K_c$ with pattern $K_p$. In this paper we consider the concept of satellite knot allowing the companion to be a \textit{handlebody-knot}, that is an embedded handlebody of genus $g$ in $S^3$. A \textit{spine} $\gamma$ of a handlebody-knot $\Gamma$ is a graph embedded in $S^3$ with $\Gamma$ a regular neighborhood.
In this section, we describe a method to construct a knot exterior with meridional essential surfaces with an arbitrarily high number of boundary components, that we will use to prove Theorem \ref{main}. The method consists of defining a specific knot used as the pattern for a satellite operation function.\\

\noindent Let $J$ be a prime knot as in the main theorem of \cite{Nogueira1}, that is with meridional essential surfaces of any positive genus and two boundary components. The knot $J$ is obtained by identifying the boundaries of two particular solid tori, say $H_1$ and $H_2$, attaching meridian to longitude, and by identifying the boundaries of the respective essential arc each contains. Denote by $X$ the torus obtained from the identified boundaries of these solid tori and by $O$ a disk in $X$ containing $X\cap J$. We isotope two copies of $X-O$ slightly to each side separated by $X$ and denote by $X_1$ and $X_2$ the resulting copies of $X$. The tori $X_1$ and $X_2$ intersect at $O$ and each bounds a solid torus, ambient isotopic to $H_1$ and $H_2$ respectively. The union of these solid tori along $O$ defines a genus two handlebody-knot $H$ with spine as in Figure \ref{H}.

\begin{figure}[htbp]
\labellist
\small \hair 2pt
\pinlabel $O$ at 36 6
\endlabellist
\centering
\includegraphics[width=0.2\textwidth]{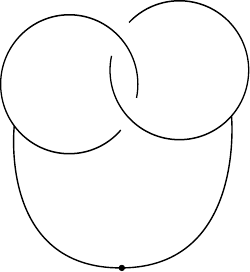}
\caption{: The spine of the handlebody-knot $H$, and the respective representation of $O$ by a point.}
\label{H}
\end{figure}

\noindent Consider a ball $B$ disjoint from $O$ such that $B^c$ intersects $H$ at a cylinder containing $O$ and two parallel trivial arcs from $J$. Note that the $2$-string tangle $(B, B\cap J)$ is essential, otherwise the punctured torus obtained from $X$ wouldn't be essential in $E(J)$. Denote by $T$ the solid torus defined by $B\cup (B^c\cap H)$. Let $J_1$ and $J_2$ be two copies of $J$ in the respective copies of $T$, say $T_1$ and $T_2$. We isotope the two arcs of $J_i\cap (T_i-B_i)$ into the boundary of $T_i$, where $B_i$ is the copy of $B$ with respect to $T_i$. For each knot $J_i$, we consider a segment of one of these arcs and a regular neighborhood $R_i$ of it, disjoint from $J_i$ otherwise. We proceed with a connect sum of $J_1$ and $J_2$ by removing the interior of $R_1$ and attaching the exterior of $R_2$, such that the disks $T_1\cap \partial R_1$ and $T_2\cap \partial R_2$ are identified. Hence, the knot $J_1\# J_2$ is in a genus two handlebody $G$ obtained by gluing $T_1$ and $T_2$ along a disk $D$ in their boundaries. (See Figure \ref{G2}.)
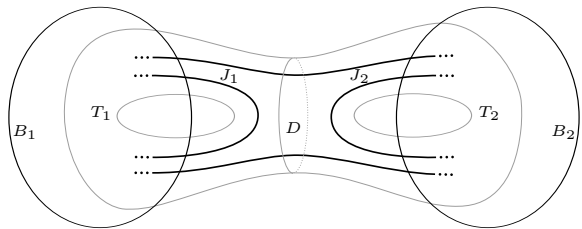
\begin{figure}[htbp]
\labellist
\tiny \hair 2pt
\pinlabel $B_1$ at 8 47
\pinlabel $T_1$ at 45 56
\pinlabel $B_2$ at 269 47
\pinlabel $T_2$ at 233 56
\pinlabel $D$ at 138 49
\pinlabel $J_1$ at 107 74
\pinlabel $J_2$ at 170 74
\endlabellist
\centering
\includegraphics[width=0.6\textwidth]{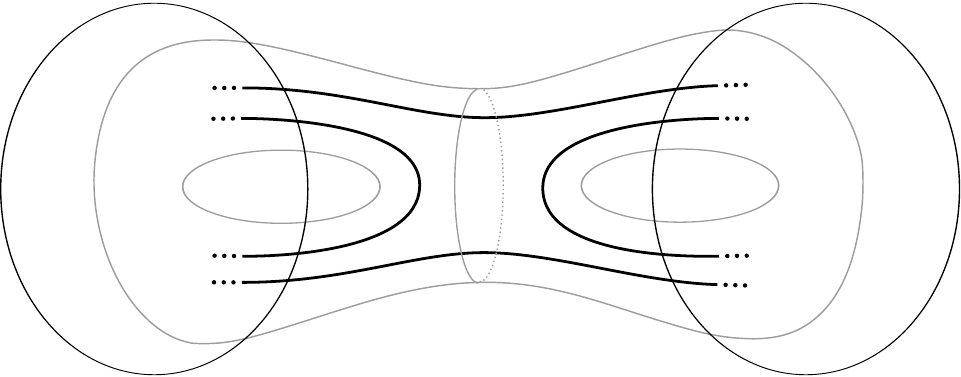}
\caption{: The handlebody $G$ together with the knot $J_1\# J_2$.}
\label{G2}
\end{figure}
As the tangle $(B, B\cap J)$ is essential and $T\cap B^c$ is a regular neighborhood of each arc of $B^c\cap J$, we have that $\partial T$ is essential in $T-J$. Moreover, from the construction of $T$ and $J$, each meridian of $T$ intersects $J$ at least twice. Hence, $\partial G$ is essential in $G-J_1\# J_2$ and, similarly, each essential disk in $G$ intersects $J_1\# J_2$ at least at two points.\\

\noindent Let $\Gamma$ be the genus $2$ handlebody-knot $4_1$, from the list in \cite{hknot}, with spine $\gamma$ as in Figure \ref{Hknot41}.
\begin{figure}[htbp]
\centering
\includegraphics[width=0.4\textwidth]{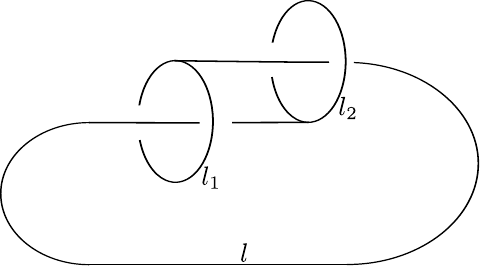}
\caption{: The spine $\gamma$ of the handlebody-knot $\Gamma$, defined by the loops $l_1$ and $l_2$ and the arc $l$.}
\label{Hknot41}
\end{figure}
Denote by $e: G\rightarrow S^3$ an embedding of $G$ into $S^3$ with image $\Gamma$, where $e(D)$ is an essential disk in a regular neighborhood $L$ of $l$.  That is, $e(D)$ is a disk that separates from $\Gamma$ two tori, $L_1$ and $L_2$, having cores $l_1$ and $l_2$, respectively, with $\Gamma=L\cup L_1\cup L_2$. (See Figure \ref{Hknot41}.) We refer to $e(J_1\# J_2)$ by $N$. The handlebody knot $\Gamma$ is embedded in a solid torus $P$ with core a trivial knot, such that there is a meridian disk of $P$ that intersects $\gamma$ at a single point in $l$. In the next definition we describe the operation used to prove Theorem \ref{main}.\\

\begin{defn}\label{satellite}
Let $\mathcal{K}$ be the set of equivalence classes of knots in $S^3$ up to ambient isotopy. For a knot $K\in \mathcal{K}$ let $h_K:P\rightarrow S^3$ be an embedding of $S^3$ such that $h_K(P)$ is a solid torus with core $K$. We define the satellite operation function $h:\mathcal{K}\rightarrow \mathcal{K}$ such that for each $K\in \mathcal{K}$ we have $h(K)=h_K(N)$.
\end{defn}

\begin{rem}\label{infinite}
The function $h$ is injective. That is, if $K_1$ and $K_2$ are distinct knots then $h(K_1)$ and $h(K_2)$ are also distinct knots. In fact, consider the tori $Y_i=\partial h_{K_i}(P)$, $i=1, 2$. The component cut by $Y_i$ from $E(h(K_i))$ containing the boundary torus of $E(h(K_i))$ is topologically the same for both knots $K_1$ and $K_2$. Hence, from the unicity of minimal $JSJ$-decompositions of compact $3$-manifolds, if two knots $h(K_1)$ and $h(K_2)$ are ambient isotopic, the tori $Y_1$ and $Y_2$ are also ambient isotopic, contradicting $K_1$ and $K_2$ being distinct. 
\end{rem}

\begin{prop}\label{prime}
For every knot $K$ the knot $h(K)$ is prime.
\end{prop}
\begin{proof}
First we observe that no ball in $G$ intersects $J_1\# J_2$ in a knotted arc, that is there is no local knot of $J_1\# J_2$ in $G$. As the knot $J_i$ is prime and the tangle $(B_i, B_i\cap J_i)$ is essential, and $(B_i^c, B_i^c\cap J_i)$ is defined by two trivial arcs, we have necessarily that there is no local knot of $J_i$ in $T_i$, and consequently there is no local knot of $J_1\# J_2$ in $G$.\\
Suppose $h(K)$ is a composite knot and consider a decomposing sphere $S$ for $h(K)$. If $S$ is disjoint from $\partial h_K(\Gamma)$ then we obtain a contradiction with the nonexistence of local knots of $J_1\# J_2$ in $G$. Then consider the intersection of $S$ with $\partial h_K(\Gamma)$ and assume that $|S\cap \partial h_K(\Gamma)|$ is minimal among all decomposing spheres for $h(K)$.\\
The sphere $S$ intersects $\partial h_K(\Gamma)$ in a collection of simple closed curves. Let $O$ be an innermost disk bounded by an innermost curve of $S\cap \partial h_K(\Gamma)$ in $S$. We have two possibilities: there is an innermost disk $O$ disjoint from $h(K)$ or an innermost disk $O$ that intersects $h(K)$ at a single point. If $O$ is disjoint from $h(K)$, as $\partial G$ is essential in $G-J_1\# J_2$ and $\partial \Gamma$ is essential in the exterior of $\Gamma$, then $\partial O$ bounds a disk in $\partial h_K(\Gamma)$. Using a ball bounded by this disk and $O$ we can isotope $S$ through $\partial h_K(\Gamma)$ in $E(h(K))$ reducing $|S\cap \partial h_K(\Gamma)|$, contradicting its minimality. If $O$ intersects $h(K)$ at a single point then $O$ is an essential disk in $h_K(\Gamma)$ intersecting $h(K)$ at a single point, which contradicts the fact that every essential disk of $G$ intersects $J_1\#J_2$ at least in two points (as observed before). Therefore, the knot $h(K)$ is prime. 
\end{proof}

Besides being prime, the knots $h(K)$ can be decomposed into two essential arcs by surfaces of genus higher than zero keeping the properties of Theorem 1 in \cite{Nogueira1} used in their construction.

\begin{prop}\label{two b}
For every knot $K$ the exterior of $h(K)$ has meridional essential surfaces of any positive genus and two boundary components. 
\end{prop}
\begin{proof}
First we note that the meridional essential surfaces of any positive genus and two boundary components $S_{g; 2}$ in $E(J)$ are in the solid torus $T$. The surface $S_{g;2}$ intersects the cylinder $B^c\cap T$ at $g-2$ annuli parallel to one string of $B^c\cap T\cap J$ and $g$ annuli parallel to the other string of $B^c\cap T\cap J$, with the latter denoted by $s$. For each copy of $T$, $T_1$ and $T_2$, denote the respective copies of $s$ by $s_1$ and $s_2$. It is convenient here to think of the connected sum $J_1\# J_2$ as made along the arcs $s_i$. We isotope the arc $s_i$ into the boundary of $T_i$, and consider a regular neighborhood $R_i$ of this segment, disjoint from $J_i$ otherwise. We consider the surfaces $S_{g;2}$ in $T_1$ and assume that the annuli of $S_{g;2}\cap T_1\cap B_1^c$ parallel to $s_1$ are in $R_1$. After the connected sum between $J_1$ and $J_2$, assumed to be along the arcs $s_1$ and $s_2$ as described before, we replace these annuli in $R_1$ by $g$ annuli in the exterior of $R_2$ parallel to the resulting arc of $J_2$ in $T_2$ (that we also denote by $J_2$). In this way, we define a new surface $S'_{g;2}$, in the handlebody $G$, obtained from $S_{g;2}$ and also with genus $g$ and two boundary components.\\
As $S'_{g;2}\cap T_2$ is a collection of annuli cutting a regular neighborhood of $J_2$ in $T_2$, there is no compressing or boundary compressing disk for $S'_{g;2}$ in $T_2$. As $S_{g;2}$ is essential in $T_1-J_1$ there is no compressing or boundary compressing disk of $S'_{g;2}$ in $T_1$. Hence, if there is a compressing or boundary compressing disk of $S'_{g;2}$ in $G$ it intersects $D$. By an outermost arc innermost curve type of argument in the compressing disk with respect to its intersection with $D$ we obtain a contradiction with the essentiality of the annuli $S'_{g;2}\cap T_2$ in $T_2$ or the essentiality of $S_{g;2}$ in $T_1-J_1$. Hence, $S'_{g;2}$ is essential in the exterior of $J_1\# J_2$ in $G$.\\
Let now $F_{g;2}$ be $e(S'_{g;2})$. We will show that $F_{g;2}$ is essential in $E(N)$. Note that $N$ is, in particular, $h(K)$ for $K$ unknotted. Suppose there is a compressing or boundary compressing disk $Q$ for $F_{g;2}$ in $E(N)$. If $Q$ is disjoint from $\partial \Gamma$ we get a contradiction with $S'_{g;2}$ being essential in the exterior of $J_1\# J_2$ in $G$. Hence, $Q$ intersects $\partial \Gamma$. Suppose $|Q\cap \partial \Gamma|$ is minimal between all compressing or boundary compressing disks of $F_{g;2}$ in $E(N)$. As $F_{g;2}$ is disjoint from $\partial \Gamma$, the disk $Q$ intersects $\partial \Gamma$ at simple closed curves. Denote by $O$ an innermost disk defined by the curves of $Q\cap \partial \Gamma$ in $Q$. As $\partial \Gamma$ is irreducible in $E(\Gamma)$, the disk $O$ cannot be essential in $E(\Gamma)$. As $J_1\# J_2$ is essential in $G$, the disk $O$ cannot be essential in $\Gamma -N$. Therefore, $\partial O$ bounds a disk in $\partial \Gamma$ which, after an isotopy of $O$ through this disk, contradicts the minimality of $|Q\cap \partial \Gamma|$. Then, $F_{g;2}$ is essential in $E(N)$.\\
For a given non-trivial knot $K$ consider $h_K(P)$ and the knot $h(K)$. Denote by $F'_{g;2}$ the surface $h_K(F_{g;2})$ in $h_K(\Gamma)$. Assume there is a compressing or boundary compressing disk $Q'$ for $F'_{g;2}$ in $E(h(K))$. In case $Q'$ is disjoint from $\partial h_K(P)$ we get a contradiction with $F_{g;2}$ being essential in $E(N)$. Then, $Q'$ intersects $\partial h_K(P)$. Suppose that $|Q'\cap \partial h_K(P)|$ is minimal between all compressing or boundary compressing disks of $F'_{g;2}$. Denote also by $O'$ an innermost disk defined by $Q'\cap \partial h_K(P)$ in $Q'$. As $\partial h_K(P)$ is essential in $E(h_K(P))$, the disk $O'$ cannot be essential in $E(h_K(P))$. As $N$ is essential in $P$ (from the construction of $P$), the disk $O'$ also cannot be essential in $P-h(K)$. Then, $\partial O'$ bounds a disk in $\partial h_K(P)$ and, as before, we get a contradiction with the minimality of $|Q'\cap \partial h_K(P)|$.\\
In conclusion, for any knot $K$ the knot $h(K)$ has a meridional essential surface of any positive genus and two boundary components.  
\end{proof}

\section{Proof of Theorem \ref{main}}

In this section we use the satellite operation described in Definition \ref{satellite} and the knots from the main theorem of \cite{Nogueira1} to prove Theorem \ref{main}. First, we start with the following proposition where we show that for any knot with a meridional essential surface in its exterior there is a knot with meridional essential surfaces of the same genus and with an unlimited number of boundary components.

\begin{prop}\label{boundaries}
Let $K$ be a knot with a meridional essential surface of genus $g$ and $n$ boundary components.\\
Then, the knot $h(K)$ has a meridional essential surface of genus $g$ and $b$ boundary components for all even $b\geq 2n$.
\end{prop}
\begin{proof}
Let $S$ be a closed surface of genus $g$ which $K$ intersects at $n$ points, corresponding to a meridional essential surface of genus $g$ and $n$ boundary components in $E(K)$, as in the statement. 
With the association of $h_K(P)$ with a regular neighborhood of $K$, we denote by $S'$ the meridional essential surface obtained from $S$ in the exterior of $h_K(P)$. Each boundary component of $S'$ bounds a meridian disk in $h_K(P)$. We consider two types of meridian disks of $h_K(P)$, with respect to $h_K(\Gamma)$,  as it intersects $h_K(\Gamma)$ at one separating disk or at two separating disks in $h_K(L)$. We refer to a meridian disk of $h_K(P)$ as \textit{type-$1$} in case it intersects $h_K(\Gamma)$ at one disk, separating $h_K(\Gamma)$ into two components. We refer to a meridian disk as of \textit{type-$2$} in case it intersects $h_K(\Gamma)$ at two disks, separating $h_K(\Gamma)$ into three components: one being a cylinder in $h_K(L)$, and each of the other two containing either $h_K(L_1)$ or $h_K(L_2)$.\\

\noindent In what follows we construct the surfaces used to prove the statement of this proposition. We define $S_{n+i}$, $i=0, 1, \ldots, n$, as a surface obtained from $S'$ by capping off its boundaries with $i$ meridians of type-$2$ and $n-i$ meridians of type-$1$. Hence, $S_{n+i}$ has genus $g$ and, in the exterior of $h_K(N)$, has $2(n+i)$ boundary components.\\ 
To proceed, we consider the surface $S_{2n}$, that intersects $h_K(P)$ at $n$ meridian disks of type-$2$, $D_1, \ldots, D_n$, ordered by index, such that $D_1\cup D_n$ cuts a cylinder from $h_K(P)$ containing all the other disks $D_i$. Let $D$ be a meridian disk of type-$1$. We assume that $D\cup D_1$ cuts a cylinder $Q_{L_1}$ from $h_K(P)$ containing $h_K(L_1)$, and $D_n\cup D$ cuts a cylinder $Q_{L_2}$ from $h_K(P)$ containing $h_K(L_2)$. Let $\partial^* Q_{L_2}$ be the annulus of intersection of $\partial Q_{L_2}$ with $\partial h_K(P)$. Let $O$ be the annulus of intersection of $Q_{L_1}$ with $\partial h_K(\Gamma)$ (the component that is disjoint from $h_K(L_1)$). Let $A$ be the annulus obtained by the union of $O$ with $D\cap E(h_K(\Gamma))$ and also with $\partial^* Q_{L_2}$. We define the surface $S_{2n+j}$, $j\geq 1$, as follows. Start with the surface $S_{2n}$. We consider $j$ meridians of type-$2$ in sequence after $D_n$ and denoted $D_{n+1}, D_{n+2}, \ldots, D_{n+j}$. We consider also $j$ copies of $A$ in $h_K(P)$ denoted, from the outside to the inside, by $A_1, A_2, \ldots, A_j$. First we consider the disk $O_p$ obtained by capping off the annulus $A_p$ by the disk that one of its boundary components bounds in $D_{n+p}$, for all $p=1, \ldots, j$. We continue by extending, in parallel to $h_K(L)$, the boundary of $O_p$ until it reaches $D_p$, and we surger the disk bounded by $O_p\cap D_p$ in $D_p$ by $O_p$. The resulting surface is the surface $S_{2n+j}$, which has genus $g$ and $2\times(2n+j)$ boundary components.

\begin{lem}\label{i<n}
The surfaces $S_{n+i}$, for $i\in \{0, 1, \ldots, n\}$, are essential in $E(h(K))$.
\end{lem}
\begin{proof}[Proof of Lemma \ref{i<n}]
Suppose a surface $S_{n+i}$, for some $i\in \{0, 1, \ldots, n\}$, is not essential and denote by $D$ a compressing disk or boundary compressing disk of $S_{n+i}$ in $E(h(K))$. Assume that $|D\cap \partial h_K(P)|$ is minimal among all compressing and boundary compressing disks of $S_{n+i}$.\\          
If $D$ intersects $\partial h_K(P)$ in some simple closed curve let $\delta$ be an innermost one in $D$ bounding an innermost disk $\Delta$. As $N$ is essential in $P$ the disk $\Delta$ cannot be essential in $h_K(P)$. As a non-trivial knot exterior in $S^3$ is boundary irreducible, the disk $\Delta$ cannot be essential in the exterior of $h_K(P)$. Hence, $\delta$ bounds a disk in $\partial h_K(P)$ and by an isotopy of $\Delta$ through this disk we can reduce $|D\cap \partial h_K(P)|$, contradicting its minimality. Then, $D$ doesn't intersect $\partial h_K(P)$ in simple closed curves.\\
Assume now that $D$ intersects $\partial h_K(P)$ in some arc. As $h(K)$ is disjoint from $\partial h_K(P)$ the arcs of $D\cap \partial h_K(P)$ have both ends in $D\cap S_{n+i}$, even when $D$ is a boundary compressing disk. Denote by $\delta$ an outermost arc of $D\cap \partial h_K(P)$ in $D$, cutting an outermost disk $\Delta$ from $D$ with boundary $\delta$ union with an arc in $D\cap S_{n+i}$. (This latter condition is not always true for all outermost disks as $\Delta$. In fact, when $D$ is a boundary compressing disk, an outermost disk $\Delta$ might include in its boundary an arc in $\partial E(h(K))$, but an outermost disk in its complement in $D$ has the desired property.) If $\delta$ has both ends in the same disk component $D_j$ of $h_K(P)\cap S_{n+i}$, by cutting and pasting along the disk cut by $\delta$ and $\partial D_j$ from $\partial h_K(P)$, we can reduce $|D\cap \partial h_K(P)|$, contradicting its minimality. Hence, $\delta$ has ends in different components of $h_K(P)\cap S_{n+i}$. As $h_K(P)\cap S_{n+i}$ is a collection of disks, $\Delta$ cannot be in $h_K(P)$. Consequently, $\Delta$ is in the exterior of $h_K(P)$ implying that $S$ is boundary compressible in $E(K)$, which contradicts its essentiality. Hence, the surfaces $S_{n+i}$, for $i\in \{0, \ldots, n\}$, are essential in $E(h(K))$.
\end{proof}

\begin{lem}\label{i>n}
The surfaces $S_{n+i}$, for $i>n$, are essential in $E(h(K))$.
\end{lem}
\begin{proof}[Proof of Lemma \ref{i>n}]
For the proof of this lemma we use branched surface theory based on work of Oertel \cite{Oertel} and Floyd and Oertel \cite{Floyd-Oertel} that we review concisely over the next paragraphs.\\
A \textit{branched surface} $B$ with generic branched locus in a $3$-manifold $M$ is a compact space locally modeled on Figure \ref{branchedmodel}(a).
\begin{figure}[htbp]

\labellist
\small \hair 0pt
\pinlabel (a) at 3 -5

\pinlabel (b) at 173 -5

\pinlabel (c) at 335 -5

\pinlabel  $\partial \text{ }N$ at 207 23
\pinlabel \tiny $h$ at 205 21

\pinlabel  $\partial \text{ }N$ at 155 40
\pinlabel \tiny $v$ at 153 37

\pinlabel $w_3=w_2+w_1$ at 383 50
\pinlabel $w_1$ at 407 14
\pinlabel $w_2$ at 391 30
\pinlabel $w_3$ at 368 30

\endlabellist
\centering
\includegraphics[width=0.9\textwidth]{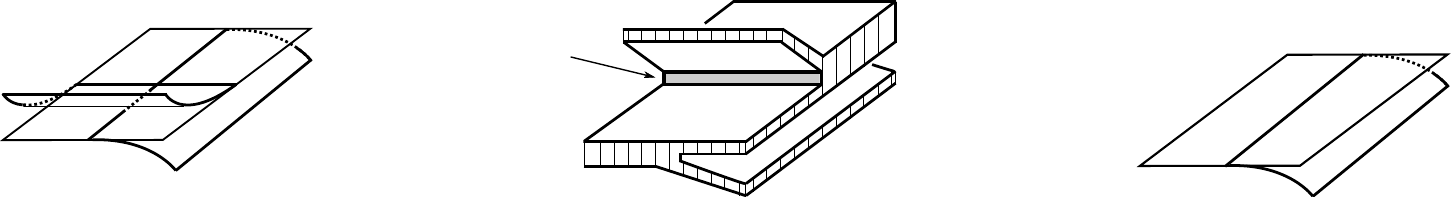}
\caption{: Local model for a branched surface, in (a), its regular neighborhood, in (b), and branch equations, in (c).}
\label{branchedmodel}
\end{figure}
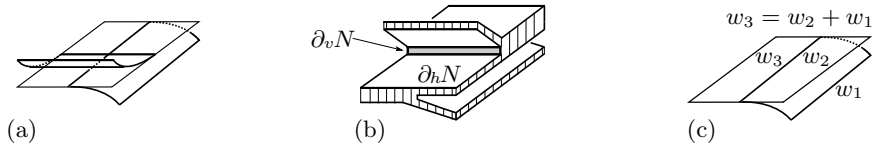
We denote by $N(B)$ a fibered regular neighborhood of $B$ embedded in $M$, locally modelled on Figure \ref{branchedmodel}(b). The boundary of $N(B)$ is the union of three compact surfaces $\partial_h N(B)$, $\partial_v N(B)$ and $\partial M\cap \partial N(B)$, where a fiber of $N(B)$ meets $\partial_h N(B)$ transversely at its endpoints and either is disjoint from $\partial_v N(B)$ or meets $\partial_v N(B)$ in a closed interval in its interior. We say that a surface $S$ is \textit{carried} by $B$ if it can be isotoped into $N(B)$ so that it is transverse to the fibers. If we associate a weight $w_i\geq 0$ to each component on the complement of the branch locus in $B$ we say that we have an invariant measure provided that the weights satisfy branch equations as in Figure \ref{branchedmodel}(c). Given an invariant measure on $B$ we can define a surface carried by $B$, with respect to the number of intersections between the fibers and the surface. We also note that if all weights are positive we say that $S$ is carried with \textit{positive weights} by $B$, which is equivalent to $S$ being transverse to all fibers of $N(B)$.\\
A \textit{disk of contact} is a disk $D$ embedded in $N(B)$ transverse to fibers and with $\partial D$ in $\partial_v N(B)$. A \textit{half-disk of contact} is a disk $D$ embedded in $N(B)$ transverse to fibers with $\partial D$ being the union of an arc in $\partial M\cap \partial N(B)$ and an arc in $\partial_v N(B)$. A \textit{monogon} in the closure of $M-N(B)$ is a disk $D$ with $D\cap \partial N(B)=\partial D$ which intersects $\partial_v N(B)$ in a single fiber. (See Figure \ref{monogon}.)\\

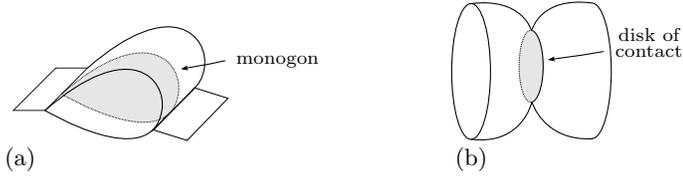
\begin{figure}[htbp]

\labellist
\small \hair 0pt
\pinlabel (a) at 3 -7

\scriptsize
\pinlabel monogon at 100 30

\small
\pinlabel (b) at 173 -7

\scriptsize
\pinlabel  \text{disk of} at 240 40

\pinlabel \text{contact} at 240 34

\endlabellist
\centering
\includegraphics{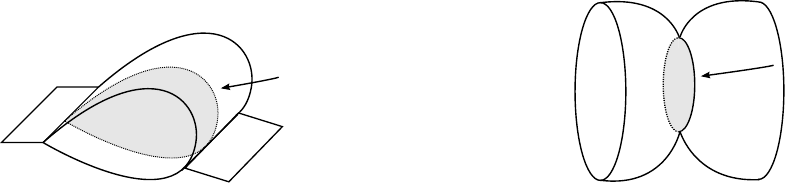}
\caption{: Illustration of a monogon and a disk of contact on a branched surface.}
\label{monogon}
\end{figure}

A branched surface $B$ embedded in $M$ is said \textit{incompressible} if it satisfies the following three properties:
\begin{itemize}
\item[(i)] $B$ has no disk of contact or half-disk of contact;
\item[(ii)] $\partial_h N(B)$ is incompressible and boundary incompressible in the closure of $M-N(B)$;
\item[(iii)] there are no monogons in the closure of $M-N(B)$.\\
\end{itemize}

Using the following theorem, by Floyd and Oertel in \cite{Floyd-Oertel}, we can determine if a surface carried by a branched surface is essential.

\begin{thm}[Floyd and Oertel, \cite{Floyd-Oertel}]\label{Floyd-Oertel}
A surface carried with positive weights by an incompressible branched surface is essential.
\end{thm}

\noindent Now we prove that the surfaces $S_{n+i}$, for $i>n$, are essential in $E(h(K))$ by showing that these surfaces are carried with  positive weights by an incompressible branched surface. Let us consider the surface $S'$ and denote by $b_1, b_2, \ldots, b_n$ its boundary components in consecutive order in $\partial h_K(P)$. Denote by $Q_j$ the annulus component of $\partial h_K(P)-b_1\cup \cdots\cup b_n$ bounded by $b_j\cup b_{j+1}$. We consider the union of $S'$, the annuli $Q_j$, $j=1, \ldots, n-1$, the annulus $A$ and a type-$2$ meridian $D_1$ of $h_K(P)$ with boundary $b_1$, as in the construction of $S_{2n}$, and denote the resulting space by $B$. We smooth the space $B$ on the intersection of $S'$, $Q_j$, $A$ and $D_1$ as explained next. For each annulus $Q_j$: isotope the boundary in $b_j$ into the exterior of $h_K(P)$ and smooth it towards $b_j$; also, smooth the boundary in $b_{j+1}$ towards the exterior of $h_K(P)$. We also smooth the boundary of $D_1$, that is $b_1$, with $S'$. With respect to the annulus $A$, we smooth its boundary in $D_1$ towards $b_1$, and we isotope its boundary in $b_n$ slightly into the exterior of $h_K(P)$ and smooth it towards $b_n$. In Figure \ref{branchedsurface} we have a schematic representation of the branched surface $B$.\\

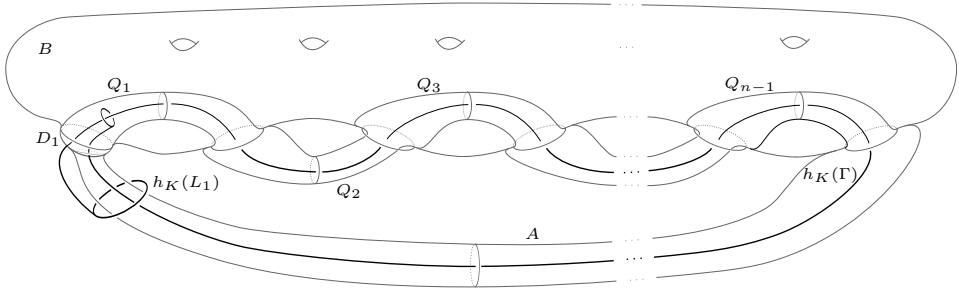
\begin{figure}[htbp]
\labellist
\tiny \hair 0pt
\pinlabel $B$ at  70 420
\pinlabel $D_1$ at  75 260
\pinlabel $Q_1$ at  200 355
\pinlabel $Q_3$ at  740 355
\pinlabel $Q_{n-1}$ at  1300 355
\pinlabel $Q_2$ at  600 170
\pinlabel $A$ at  920 95
\pinlabel $h_K(\Gamma)$ at  1440 198
\pinlabel $h_K(L_1)$ at  316 184

\endlabellist

\centering
\includegraphics[width=\textwidth]{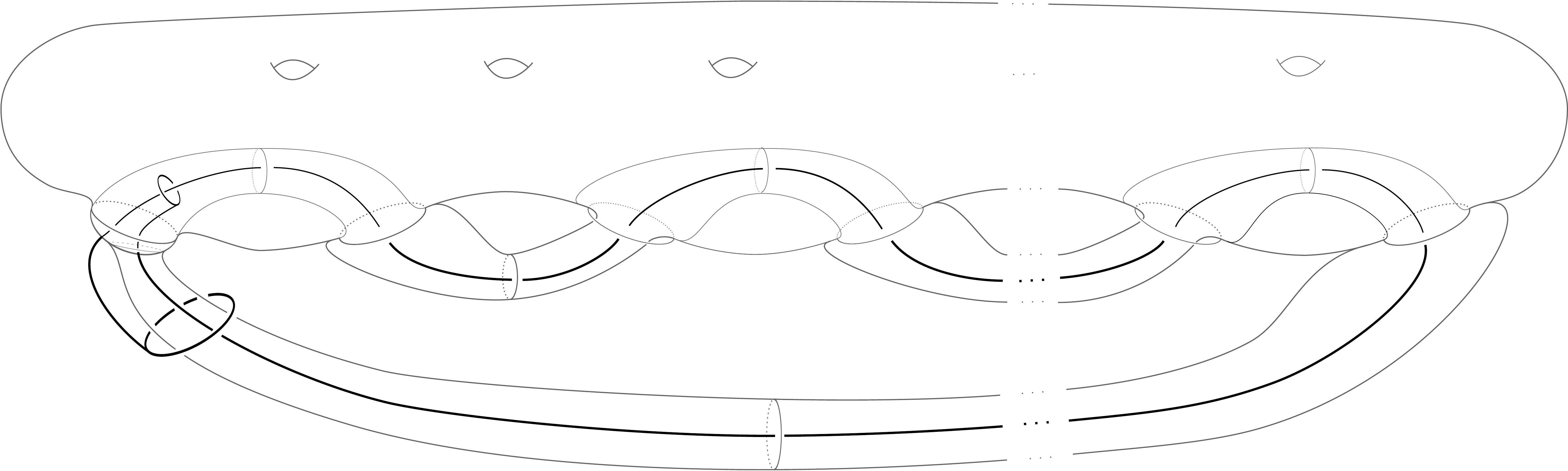}
\caption{: Schematic representation of the branched surface $B$.}
\label{branchedsurface}
\end{figure}  

Let us denote by $D_1^a$ the disk bounded by $D_1\cap A$ in $D_1$ and by $D_1^b$ the annulus defined by $D_1-D_1^a$. From the construction, the space $B$ is a branched surface with sections denoted naturally by $S'$, $Q_j$ for $j=1, \ldots, n-1$, $A$, $D_1^a$ and $D_1^b$, as illustrated in Figure \ref{branchedsurface}. We denote a regular neighborhood of $B$ by $N(B)$. The surface $S_{n+i}$ is carried with positive weights by $B$ together with the invariant measure on $B$ defined by $w_{S'}=1$, $w_{Q_j}=n-j+i$, $w_A=i$, $w_{D_1^a}=n$ and $w_{D_1^b}=n+i$, on the sections $S'$, $Q_j$, $A$, $D_1^a$, $D_1^b$, respectively. To prove that $S_{n+i}$, $i>n$, is essential in the exterior of $h(K)$ we show that $B$ is an incompressible branched surface and use Theorem \ref{Floyd-Oertel}.\\
The space $N(B)$ decomposes $E(h(K))$ into three components: a component cut from $E(h(K))$ by $S'$ and the annuli $Q_j$ with odd index that we denote by $E_1$; a component cut from $E(h(K))$ by $S'$, the annuli $Q_j$ with even index and the annulus $A$ that we denote by $E_2$; a component cut from $E(h(K))$ by $S'$, $D_1^a$, $D_1^b$, the annuli $Q_j$, $j=1, \ldots, n-1$, and $A$, that we denote by $E_p$.\\
As $E_1$ is disjoint from $\partial E(h(K))$ there are no boundary compressing disks for $\partial_h N(B)$ in $E_1$. In $\partial E_1$, the components of $\partial_v N(B)$ correspond to annuli associated to the boundary of $Q_j$ in $b_j$, for $j$ odd. Hence, if $\partial_h N(B)$ has a compressing disk in $E_1$, as it is disjoint from $\partial_v N(B)$, we can isotope its boundary into $S'$, contradicting $S'$ being essential in $E(K)$. On the other hand, a monogon disk in $E_1$ would have boundary defined by an arc in some $Q_j$ and an arc in $S'$, being a boundary compressing disk for $S'$ in $E(K)$, contradicting again $S'$ being essential.\\
In $\partial E_2$, the components of $\partial_v N(B)$ correspond to annuli associated to the boundary of $Q_j$ in $b_j$, for $j$ even, and to the annulus associated to the boundary of $A$ in $b_n$. As in the case for $E_1$, if there is a compressing disk for $\partial_h N(B)$ in $E_2$ then we get a contradiction with $S'$ being essential in $E(K)$. If there is a monogon disk in $E_2$ it would have boundary defined by an arc in some $Q_j$ or in $A$ and an arc in $S'$, defining a boundary compressing disk for $S'$ in $E(K)$, and contradicting again $S'$ being essential in $E(K)$. If there is a boundary compressing disk for $\partial_h N(B)$ in $E_2$ then an arc of such a disk boundary can be assumed to be the arc $s_1$ defined by $h_K(L_1)\cap h(K)$. The solid torus $h_K(L_1)$ is in a ball in $E_2$ intersecting $S_{n+i}$ in $D_1$, and $s_1$ has the isotopy type of the knotted arc $J_1$ in this ball. The existence of a boundary compressing disk of $\partial_h N(B)$ in $E_2$ with $s_1$ in the boundary contradicts $J_1$ being a knotted arc.\\
The component $E_p$ defines together with $E_p\cap h(K)$ a $3$-string tangle defined by a knotted arc $s_2$, with the pattern of $J_2$ in $h_K(L_2)$ and two parallel unknotted arcs in $h_K(L)$, denoted by $t_1$ and $t_2$. There is only one component of $\partial_v N(B)$ in $\partial E_p$ and it corresponds to the boundary of $A$ in $D_1$, denoted by $a$. The end points of each $t_i$, $i=1, 2$, in $E_p$ are separated by $a$ in $\partial E_p$, and the ends of $s_2$ are in the same disk bounded by $a$ in $\partial E_p$, say $D_a$. We denote the other disk bounded by $a$ in $\partial E_p$ by $D_a'$. As $a$ is separating in $\partial E_p$ there are no monogons of $\partial N(B)$ in $E_p$, and the boundary of a compressing or boundary compressing disk for $\partial_h N(B)$ in $E_p$ intersects $\partial_h N(B)$ only at $D_a$. If there is a boundary compressing disk then we can assume the arc $s_2$ is on its boundary, contradicting $s_2$ being knotted. If there is a compressing disk then it separates $s_2$ from $t_1\cup t_2$ implying that $\Gamma$ is trivial in $P$, which is a contradiction with $\Gamma$ being knotted (more exactly, the handlebody knot $4_1$ as in \cite{hknot}). This finishes the proof that $B$ is an incompressible branched surface and, consequently, from Theorem \ref{Floyd-Oertel} we also have that $S_{n+i}$ is essential in $E(h(K))$, for $i>n$. 
\end{proof}

From Lemmas \ref{i<n} and \ref{i>n} we obtain the statement of the proposition as the surfaces $S_{n+i}$, $i\geq 0$, are essential in $E(h(K))$ and $S_{n+i}$ has genus $g$ and $2n+2i$ boundary components. 
\end{proof}

This proposition offers a base for the proof of Theorem \ref{main} as follows.

\begin{proof}[Proof of Theorem \ref{main}]
Consider an infinite collection of knots $C_i$, $i\in \mathbb{N}$, as in the main theorem in \cite{Nogueira1}, that is where each of which has in its exterior a meridional essential surface $S_g$ for every genus $g\geq 0$ and two boundary components. Using the satellite operation defined in the previous section, for each knot $C_i$ we define the knot $h(C_i)$ that we denote by $K_i$. Hence, each knot $K_i$ has in its exterior, from Proposition \ref{two b}, a meridional essential surface of any positive genus and two boundary components and, from Proposition \ref{boundaries}, a meridional essential surface $S_{g;2n}$ of any genus $g$ and $2n$ boundary components for all $n\geq 2$. From Proposition \ref{prime} and Remark \ref{infinite}, the knots $h(K_i)$ are prime and pairwise distinct, and together with examples obtained after their connected sum with a non-trivial knot we complete the proof of the statement of the theorem.
\end{proof}

We now show, by proving Theorem \ref{hyperbolic}, that hyperbolic spaces can also have meridional essential surfaces of independently unbounded genus and number of boundary components.

\begin{proof}[Proof of Theorem \ref{hyperbolic}]
Consider a knot $K$ as in the statement of Theorem \ref{main}. That is, $E(K)$ contains meridional essential surfaces $S_{g;b}$ of any genus $g$ and (even) number $b$ of boundary components.
We will show that from $K$ we can construct a $3$-manifold with a hyperbolic knot whose exterior contains meridional essential surfaces $F_{g';b'}$ of genus and number of boundary components greater than or equal to $g$ and $b$, respectively.\\
From Myers \cite{Myers}, there is a null-homotopic knot $J\subset E(K)$ with hyperbolic exterior. Consider $E(J\cup K)$ and do $\frac{1}{r}$-Dehn filling on $J$ to produce a hyperbolic knot $K_r$, in a $3$-manifold $M_r$ not necessarily $S^3$ (that being the case only when $J$ is the unknot). 
As in the proof of Proposition 3.2 of \cite{Boileau-Wang} by Boileau-Wang \cite{Boileau-Wang}, there is a degree-one map $f: E(K_r)\rightarrow E(K)$, where $E(K_r)$ is the exterior of $K_r$ in $M_r$. From the construction of the degree-one map $f$, as in the proof of Proposition $3.2$ on \cite{Boileau-Wang}, we can homotope $f$ to be transverse to $S_{g;b}$ by making the immersed disk bounded by $J$ in $E(K)$, used to define $f$, transverse to $S_{g;b}$. After this homotopy of $f$, if necessary, we have that $F_{g';b'}=f^{-1}(S_{g;b})$ is a $2$-dimensional submanifold of $E(K_r)$. The restriction map $f: F_{g';b'}\rightarrow S_{g;b}$ is a degree-one map, by definition of degree-one map. Hence, by the work of Edmonds \cite{Edmonds}, it is a pinch map: there is a compact connected submanifold $F\subset F_{g';b'}$ with no more than one component of its boundary being a simple closed curve in the interior of $F_{g';b'}$, such that $f$, restricted to $F_{g';b'}$, is homotopic to the quotient map $F_{g';b'}\rightarrow F_{g';b'}/F$. In particular, this means that the genus of $F_{g';b'}$ is higher than the one of $S_{g;b}$. On top of this, as $f$ is the identity near $\partial E(K_r)$, the number of boundary components of $F_{g';b'}$ is the same as $S_{g;b}$.\\
In case $F_{g';b'}$ is incompressible, we have completed the proof. Otherwise, let $D$ be a compressing disk of $F_{g';b'}$ in $E(K_r)$. As $f:F_{g';b'}\rightarrow S_{g;b}$ is a pinch map as described above, $f(D)$ is an immersed disk in $E(K)$. In case $\partial D$ is in the pinched region of $F_{g';b'}$ by $f$, then $\partial D$ is mapped to a point of $S_{g; b}$. Hence, we can compress $F_{g';b'}$ by $D$ obtaining a surface that is still mapped by a degree-one map into $S_{g;b}$. We keep compressing until there are no more compressing disks with boundary in the pinched region. In case the induced homomorphism $f_*:\pi_1(F_{g';b'})\rightarrow \pi_1(S_{g;b})$ takes the class of $\partial D$ to the identity of $\pi_1(S_{g;b})$, we can homotope $\partial D$ into the pinched region of $F_{g';b'}$ by $f$, and repeat the previous argument. In case the class of $\partial D$ is not in the kernel of $f_*$, from the commutativity of the diagram
$$\xymatrix{\pi_1(F_{g';b'})\ar[d]_{j_*}\ar[r]^{f_*}&\pi_1(S_{g;b})\ar[d]^{i_*}\\
\pi_1(E(K_r))\ar[r]^{f_*}&\pi_1(E(K))}$$
the kernel of $i_*$ contains the class of $f_*([\partial D])$ and is non-trivial. By the Loop Theorem, there is a compressing disk of $S_{g;b}$ in $E(K)$, which is a contradiction. Hence, $F_{g';b'}$ is essential in $E(K_r)$. 
\end{proof}     


\section{Acknowledgement}
The author thanks Alan Reid for suggesting the idea of proof of Theorem \ref{hyperbolic}.


\begin{thebibliography}{99}
\bibitem{Nogueira2}
\underline{\hspace{1cm}}, The number of strings on essential tangle decompositions of a knot can be unbounded,
Alg. Geom. Top., To appear.
\end{thebibliography}
\end{document}